\documentclass[11pt, a4paper]{amsart}

\usepackage{amsfonts,amsmath,amssymb, amscd}

\newtheorem{theorem}{Theorem}[section]

\newtheorem{definition}[theorem]{Definition}
\newtheorem{prop}[theorem]{Proposition}

\theoremstyle{definition}
\newtheorem{rem}[theorem]{Remark}

\newtheorem{exas}[theorem]{Examples}

\numberwithin{equation}{section}

\newcommand\eps{\varepsilon}

\newcommand\CC{\mathbb{C}}
\newcommand\ZZ{\mathbb Z}

\DeclareMathOperator{\Hom}{Hom}
\DeclareMathOperator{\id}{id}
\DeclareMathOperator{\Id}{Id}
\DeclareMathOperator{\sgn}{sign}

\title[Polynomial identities distinguishing Galois objects]
{Examples of polynomial identities\\ 
distinguishing the Galois objects over\\ 
finite-dimensional Hopf algebras}

\author[Christian Kassel]
{Christian Kassel}
\address{Christian Kassel: 
Institut de Recherche Math\'e\-ma\-tique Avanc\'ee,
CNRS \& Universit\'e de Strasbourg,
7 rue Ren\'{e} Descartes, 67084 Strasbourg, France}
\email{kassel@math.unistra.fr}
\urladdr{www-irma.u-strasbg.fr/\raise-2pt\hbox{\~{}}kassel/}



\begin{document}

\begin{abstract}
We define polynomial $H$-identities for comodule algebras over a Hopf algebra~$H$
and establish general properties for the corresponding $T$-ideals.
In the case~$H$ is a Taft algebra or the Hopf algebra~$E(n)$,
we exhibit a finite set of polynomial $H$-identities which distinguish the Galois objects over~$H$
up to isomorphism.
\end{abstract}

\maketitle

\section{Introduction}

By the celebrated Amitsur-Levitzki theorem\,\cite{AL}, the standard polynomial of degree~$2n$ 
\begin{equation*}
S_{2n} = \sum_{\sigma \in S_n} \, \sgn(\sigma) \, X_{\sigma(1)} X_{\sigma(2)} \cdots X_{\sigma(2n)}
\end{equation*}
is a polynomial identity for the algebra~$M_n(\CC)$ of $n\times n$-matrices
with complex entries,
and $M_n(\CC)$ has no non-zero polynomial identity of degree~$< 2n$.
It follows that the identities~$S_{2n}$ distinguish the
finite-dimensional simple associative algebras over~$\CC$ up to isomorphism. 

When $G$ is an abelian group, Koshlukov and Zaicev\,\cite{KZ} established that 
any finite-dimensional $G$-graded $G$-simple associative algebra 
over an algebraically closed field of characteristic zero
is determined up to $G$-graded isomorphism by its $G$-graded polynomial identities.
Aljadeff and Haile\,\cite{AH} extended their result to non-abelian groups.
Similar results exist for other classes of algebras.

Let now $H$ be a Hopf algebra over a field~$k$. Consider the class of $H$-comodule algebras.
This class contains the $G$-graded $k$-algebras; indeed, such a algebra is nothing but
a comodule algebra over the group algebra~$kG$ equipped with its standard Hopf algebra structure. 
Similarly, a comodule algebra over the Hopf algebra
of $k$-valued functions on a finite group~$G$ is the same as a $G$-algebra, i.e., an associative $k$-algebra
equipped with a left $G$-action by algebra automorphisms.

In this context we may wonder whether the following assertion holds: 
if $H$ is a Hopf algebra over an algebraically closed field of characteristic zero,
then any finite-dimensional simple $H$-comodule algebra is determined 
up to $H$-comodule algebra isomorphism by its polynomial $H$-identities.

In this note we provide evidence in support of this assertion by means of examples.
When $H$ is the $n^2$-dimensional Taft algebra~$H_{n^2}$ or 
the $2^{n+1}$-dimensional Hopf algebra~$E(n)$, we exhibit (finitely many) polynomial $H$-identities
that distinguish the $H$-Galois objects over an algebraically closed field.
Denoting the $T$-ideal of polynomial $H$-identities for a comodule algebra~$A$ by~$\Id_H(A)$,
we deduce that $\Id_H(A) = \Id_H(A')$ implies that $A$ and $A'$ are isomorphic Galois objects.
Since each of our finite sets of identities determines the Galois object~$A$ up to isomorphism, it
also determine the $T$-ideal~$\Id_H(A)$ completely;
in a sense which we shall not make precise, these identities generate the $T$-ideal.

Before giving the explicit identities, we have to 
define the concept of a polynomial $H$-identity for a comodule algebra~$A$
over a Hopf algebra~$H$; this is done in full generality in\,\S\,\ref{sec-PI}.
When $A$~is obtained from~$H$ by twisting its product with the help of a two-cocycle,
we produce in\,\S\,\ref{ssec-detect} a universal map detecting all polynomial $H$-identities for~$A$, i.e., 
a map whose kernel is exactly the $T$-ideal~$\Id_H(A)$. 

In \S\,\ref{sec-Taft} we deal with the Taft algebra~$H_{n^2}$.
After recalling the classification of its Galois objects, we show that the degree~$2n$ polynomial 
\begin{equation*}
(YX - qXY)^n - (1-q)^n X^nY^n + (1-q)^n c \, E^nX^n
\end{equation*}
is a polynomial $H_{n^2}$-identity  and that
it distinguishes the isomorphism classes of the Galois objects over an algebraically closed field. 
In~\S\,\ref{monomial} we extend this to certain monomial Hopf algebras.

We prove a similar result for the Hopf algebra~$E(n)$ in \S\,\ref{sec-En}, 
exhibiting a finite set of polynomial $E(n)$-identities which distinguishes the Galois objects over~$E(n)$.

\section{Polynomial identities for comodule algebras}\label{sec-PI}

This is a general section in which we define polynomial identities for comodule algebras
and state general properties of the corresponding $T$-ideals.

We fix a ground field~$k$ over which all our constructions will be defined. 
In particular, all linear maps are supposed to be $k$-linear
and unadorned tensor product symbols~$\otimes$ mean tensor products over~$k$.
Throughout the paper we assume that~$k$ is \emph{infinite}.

\subsection{Reminder on comodule algebras}\label{ssec-prel}

We suppose the reader familiar with the language of Hopf algebra, as presented for instance in~\cite{M2, Sw}.
As is customary, we denote the coproduct of a Hopf algebra by~$\Delta$, its counit by~$\eps$,
and its antipode by~$S$.
We also make use of a Heyneman-Sweedler-type notation 
for the image 
\[
\Delta(x) = x_1 \otimes x_2
\]
of an element~$x$ of a Hopf algebra~$H$ under its coproduct.

Recall that a (right) $H$-\emph{comodule algebra} over a Hopf $k$-algebra~$H$
is an associative unital $k$-algebra~$A$ 
equipped with a right $H$-comodule structure whose (coassociative, counital) \emph{coaction}
\[
\delta : A \to A \otimes H
\] 
is an algebra map.
The subalgebra~$A^H$ of \emph{coin\-var\-iants} of an $H$-comodule algebra~$A$
is defined by
\begin{equation*}
A^H = \{ a \in A \, | \, \delta(a)  = a \otimes 1\} \, .
\end{equation*}

A \emph{Galois object} over~$H$ is an $H$-comodule algebra~$A$ such that $A^H = k\, 1_A$
and the map $\beta: A \otimes A \to A \otimes H$ given by $a\otimes a' \mapsto (a\otimes 1) \, \delta(a')$
($a,a'\in A$) is a linear isomorphism. For more on Galois objects, see\,\cite[Chap.~8]{M2}.

Let us now concentrate on a special class of Galois objects,
which we call \emph{twisted comodule algebras}.
Recall that a \emph{two-cocycle}~$\alpha$ on~$H$ is 
a bilinear form $\alpha : H \times H \to k$ satisfying the cocycle condition
\begin{equation*}\label{cocycle}
\alpha(x_1,y_1)\, \alpha(x_2 y_2, z)
= \alpha(y_1, z_1)\, \alpha(x, y_2 z_2)
\end{equation*}
for all $x,y,z \in H$.
We assume that $\alpha$ is invertible (with respect to the convolution product)
and normalized; the latter means that 
$\alpha(x,1)  = \alpha(1,x) = \varepsilon(x)$ for all $x\in H$.

Let $u_H$ be a copy of the underlying vector space of~$H$.
Denote the identity map~$u$ from $H$ to~$u_H$ by $x \mapsto u_x$ ($x\in H$).
We define the {algebra} ${}^{\alpha} H$ as the vector space~$u_H$ equipped 
with the product given by
\begin{equation*}\label{twisted-multiplication}
u_x \,   u_y = \alpha(x_1, y_1) \, u_{x_2 y_2}
\end{equation*}
for all $x$, $y \in H$.
This product is associative thanks to the cocycle condition;
the two-cocycle~$\alpha$ being normalized,
$u_1$ is the unit of~${}^{\alpha} H$.

The algebra ${}^{\alpha} H$ is an $H$-comodule algebra
with coaction 
$\delta \colon {}^{\alpha} H \to {}^{\alpha} H \otimes H$
given for all $x\in H$ by
\begin{equation*}\label{twisted-coaction}
\delta (u_x) = u_{x_1} \otimes x_2 \, .
\end{equation*}
It is easy to check that the subalgebra of coinvariants of~${}^{\alpha} H$ coincides with~$k \, u_1$ 
and that the map $\beta: {}^{\alpha} H \otimes {}^{\alpha} H \to {}^{\alpha} H \otimes H$ is bijective,
turning ${}^{\alpha} H$ into a Galois object over~$H$.
Conversely, when $H$ is finite-dimensional, 
any Galois object over~$H$ is isomorphic to a comodule algebra of the form~${}^{\alpha} H$.

\subsection{Polynomial $H$-identities}\label{ssec-PI}

Let us now define the notion of a polynomial $H$-identity for an $H$-comodule algebra~$A$.
(Polynomial identities for module algebras over a Hopf algebra have been defined e.g.\ in\,\cite{BL, Be}.)

For each $i = 1, 2,\ldots $ consider a copy~$X_i^H$ of~$H$; the identity map from~$H$ to~$X_i^H$
sends an element $x\in H$ to the symbol~$X_i^x$. 
Each map~$x \mapsto X_i^x$ is linear and is determined by its values on a linear basis of~$H$. 

Now take the tensor algebra on the direct sum $X_H = \bigoplus_{i\geq 1} \, X_i^H$:
\[
T = T(X_H) = T \left(\bigoplus_{i\geq 1} \, X_i^H \right) .
\]
This algebra is isomorphic to the algebra of \emph{non-commutative} polynomials in the indeterminates~$X_i^{x_r}$,
where $i= 1, 2, \ldots$ and $\{x_r\}_r$ is a linear basis of~$H$.
The algebra~$T$ is graded with all generators~$X_i^x$ homogeneous of degree~$1$.

There is a natural $H$-comodule algebra structure on~$T$ whose coaction $\delta : T \to T \otimes H$ is given by
\begin{equation*}\label{T-coaction}
\delta(X_i^x) = X_i^{x_1} \otimes x_2 \, .
\end{equation*}
The coaction obviously preserves the grading.

\begin{definition}
An element $P \in T$ is a polynomial $H$-identity for the $H$-comodule algebra~$A$
if $\mu(P) = 0$ for all $H$-comodule algebra maps $\mu : T \to A$.
\end{definition}

\begin{exas}
(a) When $H$ is the trivial one-dimensional Hopf algebra~$k$, then a $H$-comodule algebra~$A$
is nothing but an associative unital algebra.
In this case a polynomial $H$-identity for~$A$ is a classical polynomial identity, 
i.e., a non-commutative polynomial $P(X_1, X_2, \ldots)$
such that $P(a_1, a_2, \ldots) = 0$ for all $a_1, a_2, \ldots \in A$.

(b) When $H= kG$ is a group algebra, a polynomial $H$-identity is the same as a 
$G$-graded polynomial identity, as defined for instance in~\cite{BZ}.

(c) Let $H$ be an arbitrary Hopf algebra and $A$ an $H$-comodule algebra.
Assume that the subalgebra~$A^H$ of coinvariants is central in~$A$ 
(such a condition is satisfied e.g.\ when $A = {}^{\alpha} H$ is a twisted comodule algebra).
For $x,y \in H$ consider the following elements of~$T$:
\begin{equation*}
P_x = X_1^{x_1} \, X_1^{S(x_2)} 
\quad\text{and}\quad
Q_{x,y} = X_1^{x_1} \, X_1^{y_1} \, X_1^{S(x_2y_2)} \, .
\end{equation*}
Then the commutators
$P_x \, X_2^z - X_2^z \, P_x$ and $Q_{x,y} \, X_2^z - X_2^z \, Q_{x,y}$
are polynomial $H$-identities for~$A$ for all $x,y,z \in H$.
Indeed, $P_x$ and $Q_{x,y}$ are coinvariant elements of~$T$ by~\cite[Lem\-ma~2.1]{AK}.
Thus for any $H$-comodule algebra map $\mu : T \to A$,
the elements $\mu(P_x)$ and~$\mu(Q_{x,y})$ are coinvariant, hence central, in~$A$.
\end{exas}

Denote the set of all polynomial $H$-identities for~$A$ by~$\Id_H(A)$.
By definition,
\begin{equation*}\label{I(A)-def}
I_H(A) = \bigcap_{\mu}\, \ker \mu \, ,
\end{equation*}
where $\mu$ runs over all $H$-comodule algebra maps $T \to A$.

\begin{prop}
The set $I_H(A)$ has the following properties:

(a) it is a graded two-sided ideal of~$T = T(X_H)$, i.e.,
\[
I_H(A) \, T \subset I_H(A) \supset T \, I_H(A)
\]
and
\[
I_H(A) = \bigoplus_{r\geq 0} \,  \left( I_H(A) \, \bigcap \, T^r(X_H) \right)  ;
\]

(b) it is a right $H$-coideal of~$T$, i.e.,
\[
\delta\bigl(I_H(A)\bigr) \subset I_H(A) \otimes H \, ;
\]

(c) any endomorphism $f: T \to T$ of $H$-comodule algebras preserves~$I_H(A)$:
\[
f \left( I_H(A) \right) \subset I_H(A) \, .
\]
\end{prop}

The proof follows the same lines as the proof of~\cite[Prop.~2.2]{AK}. 
Note that the assumption that $k$ is infinite is needed to establish that the ideal~$I_H(A)$ is graded.
We can summarize Property\,(c) by saying that $I_H(A)$ is a $T$-ideal, 
a standard concept in the theory of polynomial identities (see~\cite{Ro}).

It is also clear that, if $A \to A'$ is a map of $H$-comodule algebras, then 
\[
I_H(A) \subset I_H(A') \, .
\]
In particular, if $A$ and $A'$ are isomorphic $H$-comodule algebras, then 
\[
I_H(A) = I_H(A') \, .
\]

In \S\S\,\ref{sec-Taft}--\ref{sec-En} we will consider certain finite-dimensional Hopf algebras~$H$ such that 
the equality $I_H(A) = I_H(A')$ for twisted comodule algebras $A$, $A'$
implies that $A$ and $A'$ are isomorphic.

To this end, we next show how to detect polynomial $H$-identities for twisted comodule algebras.

\subsection{Detecting polynomial identities}\label{ssec-detect}

Let ${}^{\alpha} H$ be a twisted comodule algebra for some normalized convolution invertible two-cocycle~$\alpha$,
as defined in\,\S\,\ref{ssec-prel}. 
We claim that the polynomial $H$-identities for~${}^{\alpha} H$ can be detected 
by a ``universal'' comodule algebra map
\[
\mu_{\alpha}: T \to S \otimes {}^{\alpha} H \, ,
\]
which we now define.

For each $i= 1,2, \ldots$, consider a copy $t_i^H$ of~$H$, identifying $x\in H$ linearly with the symbol~$t_i^x \in t_i^H$,
and define $S$ to be the symmetric algebra on the direct sum $t_H = \bigoplus_{i\geq 1} \, t_i^H$:
\[
S = S(t_H) = S \left(\bigoplus_{i\geq 1} \, t_i^H \right) .
\]
The algebra~$S$ is isomorphic to the algebra of \emph{commutative} polynomials in the indeterminates~$t_i^{x_r}$,
where $i = 1, 2, \ldots$ and $\{x_r\}_r$ is a linear basis of~$H$.

The map $\mu_{\alpha}: T \to S \otimes {}^{\alpha} H$ is given by
\begin{equation}\label{mu}
\mu_{\alpha}(X_i^x) = t_i^{x_1} \otimes u_{x_2} \, .
\end{equation}

The algebra~$S \otimes {}^{\alpha}H$ is generated by the symbols $t_i^x u_y$ ($x,y \in H; i\geq 1)$ as a $k$-algebra
(we drop the tensor product sign~$\otimes$ between the $t$-symbols and the $u$-symbols). 
It is an $H$-comodule algebra whose $S(t_H)$-linear coaction extends the coaction of~${}^{\alpha}H$:
\begin{equation*}
\delta(t_i^x u_y) = t_i^x u_{y_1} \otimes y_2 \, .
\end{equation*}

It is easy to check that $\mu_{\alpha}: T \to S \otimes {}^{\alpha} H$ is an $H$-comodule algebra map.
Its \emph{raison d'\^etre} becomes clear in the following statement.

\begin{theorem}\label{detect}
An element $P \in T$ is a polynomial $H$-identity for~${}^{\alpha} H$ 
if and only if $\mu_{\alpha}(P) = 0$; equivalently,
$I_H({}^{\alpha} H) = \ker \mu_{\alpha}$.
\end{theorem}

To prove Theorem~\ref{detect} we need the following proposition.

\begin{prop}\label{lem-mu-univ}
For every $H$-comodule algebra map $\mu : T \to {}^{\alpha} H$, 
there is a unique algebra map $\chi : S \to k$ such that 
\[
\mu = (\chi \otimes \id) \circ \mu_{\alpha} \, .
\]
\end{prop}

\begin{proof}
By the universal property of the tensor algebra~$T$,
the set of $H$-comodule algebra maps $T \to {}^{\alpha} H$ is naturally in bijection
with the vector space $\Hom^H(X_H,{}^{\alpha} H)$ of $H$-colinear maps from~$X_H$ to~${}^{\alpha} H$.
Now, since the isomorphism $u: H \to u_H = {}^{\alpha} H$ is a comodule map, we have a natural identification
$\Hom^H(X_H,{}^{\alpha} H) \cong \Hom^H(X_H,H)$.

On the other hand, by the universal property of the symmetric algebra~$S$, the set of algebra maps $S \to k$ is in bijection
with the vector space $\Hom(t_H,k)$ of linear maps from~$t_H$ to~$k$.

Recall a basic fact from ``colinear algebra'': for any right $H$-comodule~$M$ with coaction $\delta: M \to M \otimes H$,
the linear map $\Hom^H(M,H) \to \Hom(M,k)$ given by $\mu \mapsto \eps \circ \mu$ is an isomorphism
with inverse $\chi \mapsto (\chi \otimes \id) \circ \delta$.
As a consequence, any $H$-comodule map $\mu : M \to H$ is necessarily 
of the form $\mu = (\chi \otimes \id) \circ \delta$, where $\chi = \eps \circ \mu$.

Combining these observations yields a proof of the proposition.
We can be even more precise: the algebra map $\chi : S \to k$ uniquely associated to~$\mu$ 
in the statement is determined on the generators~$t_i^x$ by $\chi(t_i^x) = \eps (u^{-1}(\mu(X_i^x))$.
\end{proof}

\begin{proof}[Proof of Theorem~\ref{detect}]
Let $P \in T$ be in the kernel of~$\mu_{\alpha}$. Since by Proposition\,\ref{lem-mu-univ}
any $H$-comodule algebra map $\mu : T \to {}^{\alpha} H$ is of the form
$\mu = (\chi \otimes \id) \circ \mu_{\alpha}$, it follows that $\mu(P)=0$.
This implies $\ker \mu_{\alpha} \subset I_H({}^{\alpha} H)$.

To prove the converse inclusion, start from a polynomial $H$-identity~$P$
and observe that for every algebra map  $\chi : S \to k$, 
the composite map $\mu = (\chi \otimes \id) \circ \mu_{\alpha}$
from~$T$ to~${}^{\alpha} H$ is a comodule algebra map. 
By definition of a polynomial $H$-identity, we thus have $\mu(P)= 0$.
Now choose a basis~$\{x_r\}_r$ of~$H$ and expand $\mu_{\alpha}(P)$ as
$\mu_{\alpha}(P) = \sum_r \, \mu^{(r)}_{\alpha}(P) \otimes u_{x_r}$,
where $\mu^{(r)}_{\alpha}(P)$ belongs to~$S$. Then 
\begin{equation*}
0 = \mu(P) = \sum_r \, \chi(\mu^{(r)}_{\alpha}(P)) \,u_{x_r} \, .
\end{equation*}
Since the elements $u_{x_r}$ are linearly independent, we have $\chi(\mu^{(r)}_{\alpha}(P)) = 0$ for all~$r$.
This means that $\mu^{(r)}_{\alpha}(P) \in S$ vanishes under any evaluation $\chi : S \to k$;
in other words, the polynomial $\mu^{(r)}_{\alpha}(P)$ takes only zero values. 
The ground field~$k$ being infinite, this implies $\mu^{(r)}_{\alpha}(P) = 0$, and hence $\mu_{\alpha}(P) = 0$.
Therefore, $I_H({}^{\alpha} H) \subset \ker \mu_{\alpha}$.
\end{proof}

\section{Taft algebras}\label{sec-Taft}

Let $n$ be an integer~$\geq 2$ and $k$ a field whose characteristic does not divide~$n$.
We assume that $k$ contains a primitive $n$-th root of unity, which we denote by~$q$.

\subsection{Galois objects over a Taft algebra}\label{Taft-def}

The Taft algebra~$H_{n^2}$ has the following presentation as a $k$-algebra:
\[
H_{n^2} = k \, \langle\, x,y \,|\, x^n = 1 \, , \; yx = q xy\, , \;  y^n = 0 \,  \rangle 
\]
(see~\cite{Tf}). The set $\{x^iy^j\}_{0 \leq i,j \leq n-1}$ is a basis of the 
vector space~$H_{n^2}$, which therefore is of dimension~$n^2$. 

The algebra~$H_{n^2}$ is a Hopf algebra with coproduct~$\Delta$, counit~$\eps$ and antipode~$S$ defined by
\begin{eqnarray}\label{coproduct}
\Delta(x) = x \otimes x\, ,& \quad & \Delta(y) = 1 \otimes y + y \otimes x\, ,\\
\eps(x) = 1\, ,& \quad & \eps(y) = 0 \, ,\\
S(x) = x^{-1} = x^{n-1} \, ,& \quad & S(y) = - yx^{-1} = - q^{-1} x^{n-1} y \, .
\end{eqnarray}
When $n=2$, this is the four-dimensional Sweedler algebra.

Given scalars $a, c \in k$ such that $a \neq 0$, we consider the algebra~$A_{a,c}$ 
with the following presentation:
\[
A_{a,c} = k \, \langle\, x,y \,|\, x^n = a \, , \; yx = q xy \, , \; y^n = c  \, \rangle \, . 
\]
This is a right $H_{n^2}$-comodule algebra with coaction given by the same formulas as\,\eqref{coproduct}.

By~\cite[Prop.\,2.17 and Prop.\,2.22]{Ma1} (see also \cite{DT2}), any Galois object over~$H_{n^2}$ is isomorphic
to~$A_{a,c}$ for some scalars $a$, $c$ with $a\neq 0$. 
Moreover, $A_{a,c}$ is isomorphic to $A_{a',c'}$ as a comodule algebra
if and only if there is $v\in k^{\times} = k - \{0\}$ such that
$a' = v^n a$ and $c' = c$.
It follows that $(a,c) \mapsto A_{a,c}$ induces a bijection between $k^{\times}/(k^{\times})^n \times k$
and the set of isomorphism classes of Galois objects over~$H_{n^2}$. 
(Note that $k^{\times}/(k^{\times})^n$ is isomorphic to the cohomology group~$H^2(G,k^{\times})$.)

If $k$ is algebraically closed, then $k^{\times} = (k^{\times})^n$
and any Galois object over~$H_{n^2}$ is isomorphic to~$A_{1,c}$ for a unique scalar~$c$.

\subsection{A polynomial identity distinguishing the Galois objects}\label{Taft-PI}

Let $A = A_{a,c}$ be a Galois object as defined in \S\,\ref{Taft-def}.
Such a comodule algebra is a twisted comodule algebra~${}^{\alpha} H_{n^2}$ for some 
normalized convolution invertible two-cocycle~$\alpha$. 
It can be checked that the map $u: H_{n^2} \to A_{a,c}$ is such that $u_1 = 1$, $u_x = x$ and $u_y = y$.
This allows us to compute the corresponding universal comodule algebra map
$\mu_{\alpha} : T \to S \otimes A_{a,c}$ on certain elements of~$T$.

For simplicity, we set 
$E = X_1^1$, $X= X_1^x$, $Y= X_1^y$ for the $X$-symbols, and 
$t_1 = t_1^1$, $t_x= t_1^x$, $t_y= t_1^y$ for the $t$-symbols. 
In view of\,\eqref{mu} and\,\eqref{coproduct},
we have
\begin{equation}\label{mu-Taft}
\mu_{\alpha}(E) = t_1 \, , \quad \mu_{\alpha}(X) = t_x x\, , \quad \mu_{\alpha}(Y) = t_1 y + t_y x \, .
\end{equation}
(In the previous formulas we consider the commuting $t$-variables as extended scalars;
this allows us to drop the unit~$u_1$ of~$A_{a,c}$ and the tensor symbols between the $t$-variables 
and the $u$-variables.)

\begin{prop}\label{mainth}
The degree~$2n$ polynomial
\[
P_c = (YX - qXY)^n - (1-q)^n X^nY^n +  (1-q)^n c \,  E^nX^n 
\]
is a polynomial $H_{n^2}$-identity for the Galois object~$A_{a,c}$.
\end{prop}

This polynomial is a generalization of the degree~$4$ identity
\[
(XY+YX)^2 - 4 X^2 Y^2 + 4c\, E^2 X^2
\]
obtained for the Sweedler algebra in~\cite[Cor.\,10.4]{AK}
(in the special case $b=0$).

\begin{proof}
It suffices to check that $\mu_{\alpha}(P_c)= 0$ 
using\,\eqref{mu-Taft} and the defining relations of~$A_{a,c}$.

Since $yx = qxy$, we have
\begin{eqnarray*}
\mu_{\alpha}(YX- qXY) 
& = & (t_1 y + t_y x)t_x x - q t_x x (t_1 y + t_y x) \\
& = & (1-q) t_xt_y x^2 + t_1t_x (yx - q xy) \\
& = & (1-q) t_xt_y x^2 \, .
\end{eqnarray*}
Therefore, in view of~$x^n = a$, we obtain
\begin{equation*}
\mu_{\alpha}\left( (YX- qXY)^n \right) = (1-q)^n t_x^nt_y^n x^{2n} = a^2 (1-q)^n t_x^nt_y^n \, .
\end{equation*}
We also have $\mu_{\alpha}(E^n) = t_1^n$ and $\mu_{\alpha}(X^n) = t_x^n x^n = at_x^n$. 

To compute $\mu_{\alpha}(Y^n)$, we need the following well-known fact (see \cite[Lem\-ma~2.2]{Ma1}): 
if $u$ and $v$ satisfy the relation $vu = q uv$ for some primitive $n$-root of unity~$q$, then
\begin{equation}\label{uv}
(u+v)^n = u^n + v^n \, .
\end{equation}
Since $yx = qxy$, we may apply\,\eqref{uv} to $u= t_y x$ and $v = t_1 y$.
We thus obtain
\begin{equation*}
\mu_{\alpha}(Y^n) = t_y^n x^n + t_1^n y^n 
= a t_y^n + c t_1^n \, .
\end{equation*}
Combining the previous equalities, we obtain $\mu_{\alpha}(P_c) = 0$.
\end{proof}

The following result shows that the polynomial identity of Proposition~\ref{mainth} distinguishes
the Galois objects of the Taft algebra.

\begin{theorem}\label{maincor}
If $k$ is algebraically closed, then
$\Id_{H_{n^2}}(A_{a,c}) = \Id_{H_{n^2}}(A_{a',c'})$
implies that $A_{a,c}$ and $A_{a',c'}$ are isomorphic comodule algebras.
\end{theorem}

\begin{proof}
By the last remark in \S\,\ref{Taft-def}, we may assume $a=a'=1$.
Consider the elements $P_c \in \Id_{H_{n^2}}(A_{1,c})$ and $P_{c'} \in \Id_{H_{n^2}}(A_{1,c'})$
given by Proposition\,\ref{mainth}. By the equality of $T$-ideals, both $P_c$ and $P_{c'}$ 
are polynomial $H_{n^2}$-identities for $A_{1,c}$. Hence, so is the difference $P_c - P_{c'}$.
Therefore, $\mu_{\alpha}(P_c - P_{c'}) = 0$. Now, 
\begin{equation*}
\mu_{\alpha}(P_c - P_{c'})  = \mu_{\alpha}\left( (c-c') (1-q)^n E^nX^n \right) = 
(c-c') (1-q)^n t_1 ^n t_x^n \, .
\end{equation*}
Since $(1-q)^n t_1 ^n t_x^n \neq 0$, we have $c - c' = 0$.
This implies $A_{1,c} = A_{1,c'}$.
\end{proof}

The previous proof shows that the theorem holds if we only assume
an inclusion $\Id_{H_{n^2}}(A_{a',c'}) \subset \Id_{H_{n^2}}(A_{a,c})$ of $T$-ideals.
Note also that the single polynomial $H_{n^2}$-identity~$P_c$
determines the full $T$-ideal~$\Id_{H_{n^2}}(A_{a,c})$ over an algebraically closed field.

\begin{rem}
If $k$ is \emph{not} algebraically closed, then the equality of $T$-ideals of Theorem\,\ref{maincor}
imply that $A_{a',c'}$ is a \emph{form} of~$A_{a,c}$.
Recall that an $H_{n^2}$-comodule algebra $A$ is a {form} of~$A_{a,c}$ if there is 
an algebraic extension $k'$ of~$k$ such that $k'\otimes_k A$ and $k'\otimes_k A_{a,c}$ 
are isomorphic comodule algebras over the Hopf algebra $k'\otimes_k H_{n^2}$.
\end{rem}

\subsection{Monomial Hopf algebras}\label{monomial}

Taft algebras can be generalized as follows. 
Let $G$ be a finite group and $x$ a central element of~$G$ of order~$n$.
We also assume that there exists a homomorphism $\chi: G \to k^{\times}$
such that $\chi^n = 1$ and $\chi(x) = q$ is the fixed primitive $n$-root of unity.

To these data we associate a Hopf algebra~$H$ as follows: as an algebra,
$H$~is the quotient of the free product $kG * k[y]$ by the two-sided ideal generated
by the relations
\begin{equation*}
y^n = 0
\quad\text{and}\quad
yg = \chi(g) \, gy \, . \quad (g\in G)
\end{equation*}
The elements $gy^i$, where $g$ runs over the elements of~$G$ and $i= 0, 1, \ldots, n-1$, form a basis
of~$H$, whose dimension is equal to~$n |G|$.

The algebra~$H$ has a Hopf algebra structure such that $kG$ is a Hopf subalgebra of~$H$
and
\begin{equation*}
\Delta(y) = 1 \otimes y + y \otimes x \, , \quad
\eps(y) = 0 \, , \quad S(y) = - yx^{-1} \, .
\end{equation*}
In the literature this Hopf algebra is called a \emph{monomial Hopf algebra of type~$I$}
(see\,\cite[Sect.\,7]{BC},\,\cite{CHYZ}).

When $G = \ZZ/n$ and $x$ is a generator of~$G$, then $H$ is the Taft algebra~$H_{n^2}$.
Note that for an arbitrary finite group~$G$ the inclusion $\ZZ/n \, x \subset G$ induces 
a natural inclusion $H_{n^2} \subset H$ of Hopf algebras.

Given a two-cocycle $\sigma \in Z^2(G,k^{\times})$ of the group~$G$ and a scalar $c\in k$,
we define $A_{\sigma, c}$ as the algebra generated by the symbols $u_y$ and $u_g$ for all $g\in G$
and the relations
\begin{equation}\label{rel-kG}
u_1 = 1 \, , \qquad 
u_g u_h = \sigma(g,h) \, u_{gh} \, , 
\end{equation}
\begin{equation}
u_y^n = c \, , \qquad 
u_y u_g  = \chi(g) \, u_g u_y 
\end{equation}
for all $g,h \in G$.
The algebra $A_{\sigma, c}$ is an $H$-comodule algebra with coaction given by
\begin{eqnarray}
\delta(u_y) = 1 \otimes y + u_y \otimes x 
\quad\text{and}\quad
\delta(u_g) = u_g \otimes g \, . \quad(g\in G)
\end{eqnarray}

Bichon proved that any Galois object over~$H$ is isomorphic to one of the form~$A_{\sigma, c}$. Moreover,
$A_{\sigma, c}$ and $A_{\sigma', c'}$ are isomorphic comodule algebras if and only
$c= c'$ and the two-cocycles $\sigma$ and $\sigma'$ represent the same element of the cohomology group~$H^2(G,k^{\times})$.
In other words, the map $(\sigma,c) \mapsto A_{\sigma, c}$ induces a bijection between 
$H^2(G,k^{\times}) \times k $ and the set of isomorphism classes of Galois objects over~$H$
(see\,\cite[Th.\,2.1]{Bi}).

Let now introduce the same $X$-symbols $E = X_1^1$, $X= X_1^x$, $Y= X_1^y$
as in \S\,\ref{Taft-PI}. Since $H_{n^2}$ is a Hopf subalgebra of~$H$,
we can reproduce the same computation as in the proof of Proposition\,\ref{mainth}.
It allows us to conclude that
\begin{equation}\label{PI-monomial}
(YX - qXY)^n - (1-q)^n X^nY^n +  (1-q)^n c \,  E^nX^n 
\end{equation}
is a polynomial $H$-identity for the Galois object~$A_{\sigma,c}$.

\begin{theorem}\label{mainth2}
Suppose that $k$ is algebraically closed.
If 
\[\Id_{H}(A_{\sigma,c}) = \Id_{H}(A_{\sigma',c'})\, ,
\] 
then $A_{\sigma,c}$ and $A_{\sigma',c'}$ are isomorphic comodule algebras.
\end{theorem}

\begin{proof}
Proceeding as in the proof of Theorem\,\ref{maincor}, we deduce $c=c'$
from\,\eqref{PI-monomial}.
It remains to check that $\sigma$ and $\sigma'$ represent the same element of~$H^2(G,k^{\times})$.

Consider the following diagram: 
\begin{equation*}
\begin{matrix}
0 \to & \Id_{kG}(k^{\sigma}G) & \overset{}{\longrightarrow} & T(X_{kG}) & \overset{\mu}{\longrightarrow} 
& S(t_{kG}) \otimes k^{\sigma}G \\
&\iota\downarrow& &{\iota_T}{\downarrow}& & \iota_S \downarrow \\
0 \to & \Id_{H}(A_{\sigma,c}) & \overset{}{\longrightarrow} & T(X_H) & \overset{\mu}{\longrightarrow} 
& S(t_H) \otimes A_{\sigma,c} 
\end{matrix}
\end{equation*}
Here $k^{\sigma}G$ is the twisted group algebra generated by the symbols $u_g$ ($g\in G$)
and Relations\,\eqref{rel-kG}; it is the subalgebra of~$A_{\sigma,c}$ generated by the elements~$u_g$,
where $g$ runs over all elements of~$G$.

The vertical map $\iota_T: T(X_{kG}) \to T(X_H)$ is induced by the natural inclusion $kG \to H$; it is injective.
The map $\iota_S: S(t_{kG}) \otimes k^{\sigma}G \to S(t_H) \otimes A_{\sigma,c}$ is induced by the previous natural inclusion 
and the comodule algebra inclusion $k^{\sigma}G \subset A_{\sigma,c}$; 
it sends a typical generator $t_i^g u_{g'}$ of~$S(t_{kG}) \otimes k^{\sigma}G$ to the same expression
viewed as an element of~$S(t_H) \otimes A_{\sigma,c}$.
The maps~$\mu$ are the corresponding universal comodule maps; 
the horizontal sequences are exact in view of Theorem\,\ref{detect}.
The diagram is obviously commutative. 
Hence, the restriction $\iota$ of~$\iota_T$ to~$\Id_{kG}(k^{\sigma}G)$ send the latter to $\Id_{H}(A_{\sigma,c})$
and is injective. Since $\iota_S$ is injective, we have
\begin{equation*}
\Id_{kG}(k^{\sigma}G) = T(X_{kG}) \cap \Id_{H}(A_{\sigma,c}) \, .
\end{equation*}
Consequently, the equality of the theorem implies the equality 
\begin{equation*}
\Id_{kG}(k^{\sigma}G) = \Id_{kG}(k^{\sigma'}G)
\end{equation*}
of $T$-ideals of graded identities. 
We now appeal to\,\cite[Sect.\,1]{AHN}, from which it follows
that $\sigma$ and $\sigma'$ are cohomologous two-cocycles. 
\end{proof}

\section{The Hopf algebras~$E(n)$}\label{sec-En}

We now deal with the Hopf algebras~$E(n)$ considered in\,\cite{BDG, BC, Ne, PvO}.
When $k$ is an algebraically closed field of characteristic zero, 
$E(n)$~Êis up to isomorphism the only $2^{n+1}$-dimensional pointed Hopf algebra with coradical~$k\ZZ/2$.

\subsection{Galois objects over~$E(n)$}\label{ssec-En}

Fix an integer $n\geq 1$. Assume that the field~$k$ is of characteristic~$\neq 2$.
The algebra~$E(n)$ is generated by elements $x$, $y_1, \ldots, y_n$
subject to the relations
\begin{equation*}
x^2=1 \, , \quad y_i^2=0 \, , \quad y_ix+xy_i=0 \, ,\quad y_iy_j+y_jy_i=0 
\end{equation*}
for all $i,j= 1, \ldots, n$.
As a vector space, $E(n)$ is of dimension~$2^{n+1}$.

The algebra~$E(n)$ is a Hopf algebra with coproduct~$\Delta$, counit~$\eps$ and antipode~$S$ 
determined for all $i= 1, \ldots, n$ by
\begin{eqnarray}\label{En-coproduct}
\Delta(x) = x \otimes x\, ,& \quad & \Delta(y_i) = 1 \otimes y_i + y_i \otimes x\, ,\\
\eps(x) = 1\, ,& \quad & \eps(y_i) = 0 \, ,\\
S(x) = x\, ,& \quad & S(y_i) = - y_ix \, .
\end{eqnarray}
When $n=1$, the Hopf algebra $E(n)$ coincides with the Sweedler algebra.

The Galois objects over~$E(n)$ can be described as follows.
Let $a\neq k^{\times}$, $\underline{c} = (c_1, \ldots, c_n) \in k^n$, and $\underline{d} = (d_{i,j})_{i,j = 1, \dots, n}$
be a symmetric matrix with entries in~$k$.
To this collection of scalars we associate
the algebra $A(a,\underline{c}, \underline{d})$ generated by the symbols $u$, $u_1, \ldots, u_n$
and the relations
\begin{equation}
u^2 = a\, , \quad u_i^2 = c_i \, , \quad  
uu_i + u_i u = 0\, , \quad u_iu_j + u_j u_i = d_{i,j}
\end{equation}
for all $i,j = 1, \dots, n$. It is a comodule algebra with coaction 
$\delta: A(a,\underline{c}, \underline{d}) \to A(a,\underline{c}, \underline{d}) \otimes E(n)$
given for all $i = 1, \ldots, n$ by
\begin{equation}\label{coactionEn}
\delta(u) = u \otimes x \, ,\quad \delta(u_i) = 1 \otimes y_i + u_i \otimes x \, .
\end{equation}

It follows from\,\cite[Sect.\,4]{PvO} completed by\,\cite[Sect.\,2]{Ne} 
that any Galois object over~$E(n)$ is isomorphic to a comodule algebra
of the form~$A(a,\underline{c}, \underline{d})$. Moreover, 
$A(a,\underline{c}, \underline{d})$ and $A(a',\underline{c'}, \underline{d'})$ are isomorphic Galois objects
if and only if $\underline{c} = \underline{c'}$, $\underline{d} = \underline{d'}$ and $a' = v^2a$ for some
nonzero scalar~$v$. 

Consequently, if $k$ is algebraically closed, then $A(a,\underline{c}, \underline{d})$ is isomorphic to 
$A(1,\underline{c}, \underline{d})$ for any $a\neq 0$, 
and the Galois objects $A(1,\underline{c}, \underline{d})$ and $A(1,\underline{c'}, \underline{d'})$ are isomorphic
if and only if $\underline{c} = \underline{c'}$ and $\underline{d} = \underline{d'}$.

\subsection{Two families of polynomial identities}\label{ssec-PI-En}

Let us now compute the universal comodule algebra map
$\mu_{\alpha} : T \to S \otimes A(a,\underline{c}, \underline{d})$ 
corresponding to the comodule algebra~$A(a,\underline{c}, \underline{d})$.

We set $E = X_1^1$, $X = X_1^x$, $Y_i= X_1^{y_i}$ for the $X$-symbols, and 
$t_0 = t_1^1$, $t_x= t_1^x$, $t_i= t_1^{y_i}$ for the corresponding $t$-symbols. 
In view of \eqref{mu} and \eqref{coactionEn},
we have
\begin{equation}\label{mu-En}
\mu_{\alpha}(E) = t_0 \, , \quad \mu_{\alpha}(X) = t_x \, u \, , \quad \mu_{\alpha}(Y_i) = t_0\, u_i + t_i \, u 
\end{equation}
for all $i= 1, \ldots, n$.

\begin{prop}\label{prop3}
The degree~$4$ polynomials
\begin{equation*}
(XY_i+Y_iX)^2 - 4 X^2 Y_i^2 + 4c_i\, E^2 X^2 \qquad (1\leq i \leq n)
\end{equation*}
and
\begin{equation*}
2 (Y_iY_j + Y_jY_i)X^2 - (XY_i+Y_iX)(XY_j+Y_jX) -  2d_{i,j}\, E^2 X^2 \quad (1\leq i \leq j \leq n)
\end{equation*}
are polynomial $E(n)$-identities for the Galois object~$A(a,\underline{c}, \underline{d})$.
\end{prop}

\begin{proof}
In view of\,\eqref{mu-En} and of the defining relations of $A(a,\underline{c}, \underline{d})$, we obtain
\begin{equation*}
\mu_{\alpha}(E^2) =  t_0^2 \, , \quad
\mu_{\alpha}(X^2) =  at_x^2 \, , \quad
\mu_{\alpha}(Y_i^2) = a t_i^2 + c_i t_0^2 \, ,
\end{equation*}
\begin{equation*}
\mu_{\alpha}(XY_i+Y_iX) = 2 a t_x t_i \, , \quad
\mu_{\alpha}((Y_iY_j + Y_jY_i) = 2a t_i t_j + d_{i,j} t_0^2 \, .
\end{equation*}
From these equalities, it is easy to check that the above polynomials belong to the kernel of~$\mu_{\alpha}$,
hence are polynomial $E(n)$-identities.
\end{proof}

\begin{theorem}\label{mainth3}
Suppose that $k$ is algebraically closed.
If 
\[\Id_{E(n)}(A(a,\underline{c}, \underline{d})) = \Id_{E(n)}(A(a',\underline{c'}, \underline{d'}))\, ,
\] 
then $A(a,\underline{c}, \underline{d})$ and $A(a',\underline{c'}, \underline{d'})$ are isomorphic comodule algebras.
\end{theorem} 

\begin{proof}
We proceed as in the proof of Theorem\,\ref{maincor} by using the identities of Proposition\,\ref{prop3}. Note that there is 
such an identity for each scalar used to parametrize the Galois objects, and each such scalar appears
as the coefficient of the monomial~$E^2 X^2$; the latter cannot be an identity since its image under the 
universal comodule map, being equal to~$a t_0^2 t_x^2$, does not vanish.
\end{proof}

We finally note that the set of $n(n+3)/2$  polynomial $E(n)$-identities of Proposition\,\ref{prop3}
determines the $T$-ideal~$\Id_{E(n)}(A(a,\underline{c}, \underline{d}))$.

\section*{Acknowledgment}
My warmest thanks go to Eli Aljadeff who initiated me to the theory of polynomial algebras (classical and graded)
and drew my attention to the question addressed here.


\begin{thebibliography}{99}


\bibitem{AH}
E.~Aljadeff, D.~Haile,
\emph{Simple $G$-graded algebras and their polynomial identities},
arXiv:1107.4713, to appear in Trans.\ Amer.\ Math.\ Soc.

\bibitem{AHN} 
E.~Aljadeff, D.~Haile, M.~Natapov,
\emph{Graded identities of matrix algebras and the universal graded algebra},
Trans.\ Amer.\ Math.\ Soc.~362 (2010), 3125--3147.

\bibitem{AK}
E.~Aljadeff, C.~Kassel,
\emph{Polynomial identities and noncommutative versal torsors},
Adv.\ Math.~218 (2008), 1453--1495.

\bibitem{AL}
A. S. Amitsur, J. Levitzki,
\emph{Minimal identities for algebras},
Proc.\ Amer.\ Math.\ Soc.~1 (1950), 449--463. 

\bibitem{BL} 
Y. A. Bahturin, V. Linchenko,
\emph{Identities of algebras with actions of Hopf algebras},
J.~Algebra 202 (1998), 634--654.

\bibitem{BZ} 
Y. A. Bahturin, M. Zaicev, 
\emph{Identities of graded algebras},
J.~Algebra 205 (1998), 1--12.

\bibitem{BDG}
M.~Beattie, S.~D\v asc\v alescu, L.~Gr\"unenfelder, 
\emph{Constructing pointed Hopf algebras by Ore extensions},
J.~Algebra 225 (2000), 743--770. 

\bibitem{Bi}
J.~Bichon, 
\emph{Galois and bigalois objects over monomial non-semisimple Hopf algebras},
J.~Algebra Appl.~5 (2006), 653--680.

\bibitem{BC}
J.~Bichon, G.~Carnovale, 
\emph{Lazy cohomology: an analogue of the Schur multiplier for arbitrary Hopf algebras},  
J.~Pure Appl.\ Algebra~204 (2006), 627--665.

\bibitem{Be} 
A.~Berele, 
\emph{Cocharacter sequences for algebras with Hopf algebra actions}, 
J.~Algebra 185 (1996), 869--885.

\bibitem{CHYZ}
X.-W. Chen, H.-L. Huang, Y. Ye, P. Zhang, 
\emph{Monomial Hopf algebras}, 
J.~Algebra 275 (2004), 212--232.

\bibitem{DT2} Y. Doi, M. Takeuchi, 
\emph{Quaternion algebras and Hopf crossed products},
Comm.\ Algebra 23 (1995), 3291--3325.

\bibitem{KZ}
P. Koshlukov, M. Zaicev, 
\emph{Identities and isomorphisms of graded simple algebras}, 
Linear Algebra Appl.~432 (2010), 3141--3148.

\bibitem{Ma1} A. Masuoka,
\textit{Cleft extensions for a Hopf algebra generated by a nearly primitive element},
Comm.\ Algebra 22 (1994), 4537--4559.

\bibitem{M2} S. Montgomery, 
\emph{Hopf algebras and their actions on rings},
CBMS Conf.\ Series in Math., vol.~82, Amer.\ Math.\ Soc., Providence, RI, 1993.

\bibitem{Ne}
A.~Nenciu, 
\emph{Cleft extensions for a class of pointed Hopf algebras constructed by Ore extensions},
Comm.\ Algebra 29 (2001), 1959--1981.

\bibitem{PvO}
F. Panaite, F. van Oystaeyen, 
\emph{Quasitriangular structures for some pointed Hopf algebras of dimension~$2^n$}, 
Comm.\ Algebra 27 (1999), 4929--4942.

\bibitem{Ro} L. Rowen, 
\emph{Polynomial identities in ring theory}, 
Pure and Applied Mathematics,~84,
Academic Press, Inc., New York-London,~1980.

\bibitem{Sw} M.~Sweedler, 
\emph{Hopf algebras}, 
W. A. Benjamin, Inc., New York,~1969.

\bibitem{Tf}
E. J. Taft, 
\emph{The order of the antipode of finite-dimensional Hopf algebra}, 
Proc.\ Nat.\ Acad.\ Sci.\ U.S.A.\ 68 (1971), 2631--2633. 



\end{thebibliography}
\end{document}